\documentclass{article}
\usepackage[utf8]{inputenc}
\usepackage{pgf,tikz,tkz-graph}
\usetikzlibrary{arrows,shapes}
\usetikzlibrary{decorations.pathreplacing}
\usepackage{tkz-berge}
\usepackage{ytableau}
\usepackage{mathdots}
\usepackage{verbatim,subcaption}
\usepackage{fullpage}
\usepackage{hyperref,verbatim}
\usepackage{enumitem,bm}
\usepackage{url}
\usepackage{wrapfig}
\usepackage{listings}
\usepackage{amsthm,amsmath,amssymb}

\newtheorem{defi}{Definition}

\newtheorem{cor}[defi]{Corollary}
\newtheorem{thr}[defi]{Theorem}
\newtheorem{lem}[defi]{Lemma}

\newtheorem{prop}[defi]{Proposition}

\newtheorem{claim}[defi]{Claim}
\newcommand*{\myproofname}{Proof}
\newenvironment{claimproof}[1][\myproofname]{\begin{proof}[#1]}{\end{proof}}

\definecolor{light-gray}{gray}{0.95}
\definecolor{medium-gray}{gray}{0.65}
\definecolor{half-gray}{gray}{0.5}
\definecolor{dark-gray}{gray}{0.35}

\def \d{\, \mathrm{d}}
\DeclareMathOperator{\TI}{TI}
\chardef\_=`_

\title{Extremal values of degree-based entropies of bipartite graphs}
\author{
Stijn Cambie\thanks{Extremal Combinatorics and Probability Group (ECOPRO), Institute for Basic Science (IBS), Daejeon, South Korea, supported by the Institute for Basic Science (IBS-R029-C4),
E-mail: {\tt stijn.cambie@hotmail.com} or {\tt stijncambie@ibs.re.kr}.}\and
Yanni Dong\thanks{School of Mathematics and Statistics,
Northwestern Polytechnical University, P.R. China, and Faculty of Electrical Engineering, Mathematics and Computer Science, University of Twente, The Netherlands, supported by National Natural Science Foundation of China (12071370,~12131013 and~U1803263) and China Scholarship Council (202006290070),
E-mail: {\tt yndong\_math@mail.nwpu.edu.cn}.
}
\and
Matteo Mazzamurro\thanks{
Department of Computer Science,
University of Warwick, UK, supported by: EPSRC Centre for Doctoral Training in Urban Science and Progress (EP/L016400/1), EPSRC DTP (EP/N509796/1),
E-mail: {\tt  matteo.mazzamurro@warwick.ac.uk}.  }
}

\begin{document}

\maketitle

\begin{abstract}
    We characterize the bipartite graphs that minimize the (first-degree based) entropy, among all bipartite graphs of given size, or given size and (upper bound on the) order. The extremal graphs turn out to be complete bipartite graphs, or nearly complete bipartite. Here we make use of an equivalent representation of bipartite graphs by means of Young tableaux, which make it easier to compare the entropy of related graphs. We conclude that the general characterization of the extremal graphs is a difficult problem, due to its connections with number theory, but they are easy to find for specific values of the order $n$ and size $m$. We also give a direct argument to characterize the graphs maximizing the entropy given order and size.
    We indicate that some of our ideas extend to other degree-based topological indices as well.
\end{abstract}

\section{Introduction}

The first-degree based graph entropy of a graph is the Shannon entropy of its degree sequence normalized by the degree sum. In the remaining of the paper, it is just referred to as entropy.
It has attracted significant attention in organic chemistry, as measures of uniformity of a graph's structural aspect of interest \cite{mowshowitz1968entropy,bonchev1983chemical,dehmer2009entropy}. The interpretation of the entropy of a particular graph depends on knowing the extremal values.
In this paper, we continue our work~\cite{CM22+1, CM22+2}
on determining the extremal graphs (and thus extremal values) for the entropy among all graphs satisfying some natural restrictions. 
Minimizing (resp. maximizing) the entropy corresponds to maximizing (resp. minimizing) the function $h(G)=\sum_i d_i \log(d_i)$, where $(d_i)_{1 \le i \le n}$ is the degree sequence of the graph $G$,
where $G$ ranges over all members of a class of graphs. This is since the entropy $I(G)$ equals $\log(2m)-\frac 1{2m} h(G),$ where $m$ is the size (number of edges) of $G$.
Here we determine the graphs with extremal entropy among all bipartite graphs with given size $m$, or given size $m$ and (upper bound on the) order $n$.
The maximum value for the entropy is obtained by the graphs for which the degree sequence is as balanced as possible. This is a corollary of Karamata's inequality.
Note that the random graphs $G_{n,m}$ and $G_{n,p}$ are close to balanced, i.e. the maximum degree $\deg_{max}$ and the minimum degree $\deg_{min}$ are nearly the same for $m=\Theta(n^2)$, and thus have an entropy which is almost maximal. This is in line with the intuition that entropy is a measure for randomness. 

In graph theory, bipartite graphs are one of the main classes to investigate, because it finds several applications in pure and applied Mathematics. 
Hall's Matchings Theorem is a famous theorem on bipartite graphs with several applications in scheduling- and matching problems,
e.g. the kidney matching process.
When edges represent bonds between positive and negative charges, the resulting graph is bipartite. 
Furthermore, results on bipartite graphs often lead to results on more general classes of graphs. A famous example is given by the Kahn-Zhao theorem, where Kahn~\cite{Kahn01} considered a problem on the maximum number of independent sets for bipartite graphs and Zhao~\cite{Zhao10} reduced the general case to the bipartite case. 

We will show that, given the size $m$, the bipartite graphs attaining the minimum entropy are precisely the complete bipartite graphs of size $m$.
So if $m$ has $\sigma(m)$ divisors, there are $\left\lceil \frac{\sigma(m)}{2} \right\rceil$ non-isomorphic extremal graphs, all which are of the form $K_{q,y}$ with $yq=m.$

\begin{thr}\label{thr:main}
    If $G=(U \cup V,E)$ is a bipartite graph of size $m$, then $h(G) \le m \log m$. Equality occurs if and only if $G$ is a complete bipartite graph.
\end{thr}

When there is an upper bound $n$ on the order (one can extend with isolated vertices since $0\log(0)=0$), such that $m$ cannot be written as a product $yq$ with $y+q \le n$, the problem is harder.
We will prove that the extremal bipartite graphs given the size $m$ and order no more than $n$, are of the form $B(n,m,y)$ for some $1 \le y \le \sqrt{m}$.
Here $B(n,m,y)$ is the graph formed by connecting an additional vertex with $m-yq$ vertices in $V$ of $K_{q,y}=(U \cup V,E),$ where $q=\left\lfloor \frac{m}{y} \right\rfloor.$

We will also consider a different representation of bipartite graphs by means of Young tableaux, as it simplifies the description of the graph operations needed in the proofs.
For a bipartite graph $G=(U \cup V,E)$, let us write $U=\{u_1 , u_2, \ldots, u_x\}$ and $V=\{v_1, v_2, \ldots, v_y\}$ such that $\deg u_i \ge \deg u_j$ and $\deg v_i \ge \deg v_j$ whenever $i \le j.$
Let $y_i=\deg(u_i)$ and $x_j=\deg(v_j)$ for every $1\le i \le x$ and $1\le j \le y.$
We associate the tableau $T$ which contains a cell $(i,j)$ if and only if $u_iv_j \in E.$
Note that this gives a one-to-one correspondence between tableaux and bipartite graphs.

An example of an extremal bipartite graph with $n=10$ and $m=22$ has been represented in these two ways in Figure~\ref{fig:(10,22)}.
Remark that the number of cells in column $i$ corresponds to the degree of $u_i$ and the number of cells in row $j$ equals the degree of $v_j.$

\begin{figure}[h]

\begin{minipage}[b]{.49\linewidth}
\begin{center}
    \begin{tikzpicture}
    {
					
	\foreach \x in {0,1,2,3,4}{\foreach \y in {1,2,3,4}{\draw[thick] (\x,2) -- (\y,0);}}
	
	\draw[thick] (5,2) -- (3,0);
	\draw[thick] (5,2) -- (4,0);

		\foreach \x in {0,1,2,3,4,5}{\draw[fill] (\x,2) circle (0.1);}
		\foreach \x in {1,2,3,4,5,6}{\node at (\x-1,2.4) {$u_\x$};}
		\foreach \x in {1,2,3,4}{\draw[fill] (\x,0) circle (0.1);}
	    \foreach \x in {4,3,2,1}{\node at (5-\x,-0.4) {$v_\x$};}
				
	}
	\end{tikzpicture}\\
\subcaption{Graph representation}
\label{fig:Graph(10,22)}
\end{center}
\end{minipage}\quad\begin{minipage}[b]{.49\linewidth}
\begin{center}
\ytableausetup{centertableaux}
\begin{ytableau}
  \none[5] &  &  & &  &  & \none \\
  \none[5] &  &  & &  &  & \none \\
  \none[6] &  &  & &  &  & \\
  \none[6] &  &  & &  &  &  \\
  \none & \none[4] & \none[4] & \none[4] & \none[4] & \none[4] & \none[2]
\end{ytableau}\\
\subcaption{Associated Young tableau}
\label{fig:Young(10,22)}
\end{center}
\end{minipage}
\caption{Two representations of the extremal bipartite $(10,22)$-graph $B(10,22,4)$}\label{fig:(10,22)}
\end{figure}
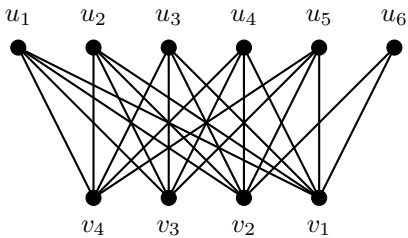
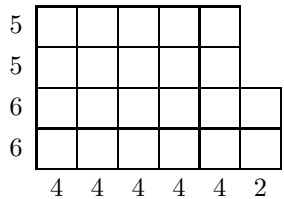

In Section~\ref{sec:bip_case}, we prove that the bipartite graphs minimizing the entropy given the size are exactly the complete bipartite graphs. We give two different approaches to prove this. One of them is a proof by induction, and the other one consists in proving that the associated tableau has to be a rectangle.
In Section~\ref{sec:denseY} we study the possible extremal bipartite graphs given order $n$ and size $m$, by proving that under certain restrictions they are of the form $B(n,m,y)$ and we estimate the value of $h$ for these graphs.
Next, in Section~\ref{sec:proofYdense}, we prove that the extremal graphs are indeed of the form $B(n,m,y).$ 
For this, we use the equivalent representation with the Young tableaux and give local operations that decrease the entropy (equivalently, such that $h$ increases).
Furthermore, in Section~\ref{sec:chemicalindex} we remark that some of the ideas can be applied to other chemical indices, also called degree-based topological indices, such as e.g. the Second Zagreb index and Reciprocal Randic index.
Just as has been done in~\cite{Cam21}, we notice that certain results hold more generally and we give the essence of the proofs of some known results. 
Here, we give short proofs for the main results in~\cite{HLP22} and~\cite{Dong21} about the graphs and bipartite graphs maximizing the entropy in the more general context of a class of chemical indices.
Finally, in Section~\ref{sec:conc}, we give some conclusions on the precise extremal graphs $B(n,m,y)$. We do this based on computational results and the estimates in Section~\ref{sec:denseY}, as well as some number theory.

\subsection{Preliminary definitions and results}\label{subsec:prelimanaryresults}

In this subsection, we define the notions and help functions we will frequently use, as well as give some basic results.

\begin{defi}
    We define the function $f(x)=x\cdot \log (x)$ 
    for $x\geq 0$. Here $\log$ denotes the natural logarithm and $f(0):=0.$
    For a graph $G$ with degree sequence $(d_i)_{1 \le i \le n}$, we define $h(G)=\sum_i f(d_i).$
    The entropy $I(G)$ of a graph of size $m$ equals $\log(2m)-\frac{1}{2m}h(G).$
\end{defi}

In the following definition of majorization and the inequality of Karamata, we only consider non-increasing sequences.

\begin{defi}
    A sequence $(x_i)_{1 \le i \le n}$ majorizes the sequence $(y_i)_{1 \le i \le n}$ if and only if  $\sum_{1 \le i \le j } x_i \ge \sum_{1 \le i \le j } y_i$ for every $1 \le j \le n$ and equality does hold for $j=n$.
\end{defi}

\begin{thr}[Karamata's inequality, \cite{Kar32}]\label{thr:Kar}
    Let $(x_i)_{1 \le i \le n}$ be a sequence majorizing the sequence $(y_i)_{1 \le i \le n}$. 
    For every convex function $g$ we have $\sum_{1 \le i \le n } g(x_i) \ge \sum_{1 \le i \le n } g(y_i).$ Furthermore this inequality is strict if the sequences are not equal and $g$ is a strictly convex function. For concave functions, the same holds with the opposite sign.
\end{thr}

A corollary of Karamata's inequality is the following

\begin{cor}\label{cor:majdeg_kar}
    Let $G$ and $G'$ be two graphs such that the degree sequence of $G'$ majorizes the degree sequence of $G$. 
    Then $h(G') \ge h(G).$ Furthermore, equality occurs if and only if the two degree sequences are equal.
\end{cor}

This was also observed by Ghalavand et al.~\cite{GEA19}.

We will also make use of difference graphs, which were introduced by Hammer et al.~\cite{Hammer90}.
The following equivalent characterization is due to Mahadev and Peled~\cite{MP95}.

\begin{thr}[\cite{MP95}]\label{thr:MP95}
    Let $G=(U\cup V, E)$ be a bipartite graph. 
    The graph $G$ is a difference graph if and only if one of the following equivalent conditions is true.
\begin{itemize}
      \item[(a)] there are no $u_1,u_2 \in U$ and $v_1,v_2\in V$ such that $u_1v_1,u_2v_2 \in E(G)$ and $u_1v_2,u_2v_1 \notin E(G)$;
      \item[(b)] every induced subgraph without isolated vertices has on each side of the bipartition a domination vertex, that is, a vertex which is adjacent to all the vertices on the other side of the bipartition. 
\end{itemize}
\end{thr}

\begin{lem}\label{lem:extr_is_diff_graph}
If $G$ is a graph maximizing $h(G)$ among all bipartite graphs of size $m$, then $G$ is a difference graph. 
\end{lem}

\begin{proof}
    Suppose that $G$ is not a difference graph. By Theorem~\ref{thr:MP95}, there are four vertices
    $u_1, u_2\in U$ and $v_1, v_2 \in V$ such that $u_1v_1, u_2v_2 \in E(G)$ and $u_1v_2, u_2v_1 \not \in E(G).$
    Assume without loss of generality that $\deg(u_1) \ge \deg(u_2)$. In that case the degree sequence of the graph $G'=\left(G\backslash \{u_2v_2\}\right) \cup \{u_1v_2\}$ will majorize the degree sequence of $G$ (strictly) and so we get the desired contradiction by Corollary~\ref{cor:majdeg_kar}.
\end{proof}

\begin{lem}\label{lem:max_with_balanced}
    Let $t,\ell>0$ be fixed positive integers.
    Under the condition that $\sum_{j=1}^x z_i =t$ and all $z_i \ge 0$ are integers, $ \sum_{j=1}^x f(z_j+\ell) -\sum_{j=1}^x f(z_j)$ is maximized when $z_1=z_2=\ldots=z_x=\left\lfloor \frac tx \right\rceil.$
\end{lem}

\begin{proof}
    Note that the function $\Delta^{\ell}(z)=f(z+\ell) - f(z)$ is a strictly concave function for every $\ell>0$ and every sequence of $x$ integers with sum $t$ do majorize the sequence with $z_1=z_2=\ldots=z_x=\left\lfloor \frac tx \right\rceil.$
    Now the result follows immediately by Karamata's inequality.
\end{proof}

For the function $\Delta^1$, we write $\Delta$ for ease of notation.
We define this one separately, as it will be used frequently.

\begin{defi}
    The function $\Delta$ is defined by $\Delta(x)=f(x)-f(x-1)=1+\int_{x-1}^{x} \log t \, \mathrm{d}t$.
    This is a strictly increasing and concave function.
\end{defi}

We also use Landau notation, such as $o()$ and $\omega()$.
A function $q(x,y)$ or expression is $o_y(p(x,y))$ or $\omega_y(p(x,y))$ if 
$\frac{q(x,y)}{p(x,y)} \to 0$ respectively $\frac{\lvert q(x,y) \rvert }{\lvert p(x,y) \rvert } \to \infty$  when $y \to \infty$.
Sometimes the dependency on $y$ is removed to keep the notation light.

We also use the notation $[k]=\{1,2,\ldots, k\}$ and $[k..\ell]=\{k,k+1,\ldots, \ell-1,\ell\}.$

\section{Minimum entropy of bipartite graphs of fixed size}\label{sec:bip_case}

In this section, we prove Theorem~\ref{thr:main}.
In both approaches, we assume that our bipartite graph is $G=(U \cup V,E)$, $U=\{u_1 , u_2, \ldots, u_x\}$ and $V=\{v_1, v_2, \ldots, v_y\}$ are the vertices with degree at least $1.$
Here we assume these are ordered, i.e. $\deg u_i \ge \deg u_j$ and $\deg v_i \ge \deg v_j$ whenever $i \le j.$
Let $y_i=\deg(u_i)$ and $x_j=\deg(v_j)$ for every $1\le i \le x$ and $1\le j \le y.$
So with this notation, we have $y=y_1$ and $x=x_1.$

\subsection{Approach 1}

\begin{proof}[Proof of Theorem~\ref{thr:main}]
    We prove the statement by induction on $m.$
    The base case $m=1$ is trivial since we only have one edge.
    So assume the statement is true when the size is at most $m-1$ and let $G$ be a graph maximizing $h(G)$ among all graphs of size $m$.
    By Lemma~\ref{lem:extr_is_diff_graph} we can assume that $G$ is a difference graph. 
    This implies that $u_1$ is a dominating vertex (dominates $V$). Let $G'=G \backslash u_1$.
    By the induction hypothesis applied to $G'$ and Lemma~\ref{lem:max_with_balanced} applied with $\ell=1$ and $z_i=x_i-1$, we know
    \begin{align*}\label{eq:upperboundh_cnG}
        h(G)&=h(G')+f(y)+\sum_{j=1}^y f(x_j) -\sum_{j=1}^y f(x_j-1)\nonumber \\ 
        &\le h(K_{1,m-y}) +f(y) + y\left(  f\left( \frac{m}{y} \right) - f\left( \frac{m}{y}-1 \right) \right).
    \end{align*}
    Equality here is only attained if $x_j=\frac{m}{y}$ for every $1\le j \le y.$
    Now the conclusion is direct as
     \begin{align*}
            h(G)\le (m-y)\log(m-y)+y\log(y)+m\left(\log(m)-\log(y)\right)-(m-y)\left(\log(m-y)-\log(y)\right)= m\log m.
    \end{align*}
    By induction, we see that equality is attained precisely for the complete bipartite graphs on $m$ edges.
\end{proof}

\subsection{Approach 2}\label{subsec:approach2}

In this subsection, we give an alternative proof for the fact that $K_{1,m}$ minimizes the entropy among all bipartite graphs of size $m$, which does not use any prerequisites.
Note that the proof can be formulated without the notion of a tableau and that we give a short proof of Lemma~\ref{lem:extr_is_diff_graph} as a claim in this notation.

In every cell $(i,j)$ of the associated tableau $T$ of a bipartite graph $G$, we put $\log(y_ix_j )=\log(x_j)+\log(y_i)$.
The sum over all cells, $h(T)=\sum_{(i,j) \in T} \log(x_j y_i)$ is now exactly equal to $\sum_j f(x_j)+\sum_i f(y_i)=h(G).$

\begin{proof}[Proof of Theorem~\ref{thr:main}]
    We first prove that the associated tableau $T$ is a Young tableau, i.e. if $(i,j) \in T$ and $0<i'\le i$ and $0<j'\le j,$ then $(i',j') \in T.$ 
    
    \begin{claim}
        If $G=(U \cup V,E)$ is a bipartite graph maximizing $h(G)$ among all bipartite graphs of size $m$, then its associated tableau is a Young tableau.
    \end{claim}
    \begin{claimproof}
        Assume $(i,j) \in T$ and $(i',j') \not\in T$.
        If $i'<i$ and $j'<j$, then $G'=G \backslash u_iv_j \cup u_{i'}v_{j'}$ satisfies 
        $h(G')=h(G)- \Delta(y_i)-\Delta(x_j)+\Delta(y_{i'}+1)+\Delta(x_{j'}+1)>h(G).$ The latter due to $\Delta$ being strictly increasing and $y_i \le y_{i'}$ and $x_j\le x_{j'}.$
        If $i'=i$ or $j'=j$, it is analogous as $h(G')-h(G)$ equals $\Delta(y_{i'}+1)-\Delta(y_i)>0$ or $\Delta(x_{j'}+1)-\Delta(x_j)>0.$
        This implies that $G$ was not extremal and so we conclude by contradiction.
    \end{claimproof}
    
    Next, we note that
    $$  \sum_{(i,j) \in T} x_j y_i \le \sum_{i,j} x_j y_i = \left(\sum_j x_j \right)\left(\sum_i y_i\right)=m^2.$$
    So by Jensen's inequality, we have 
    $$\sum_{(i,j) \in T} \log(x_j y_i) \le m \log\left(\frac{m^2}{m}\right)=m \log(m)$$ with equality if and only if $x_jy_i=m$ for every choice of $(i,j)$, i.e. $G$ is complete, bipartite.
\end{proof}

\section{Minimum entropy among dense Bipartite graphs}\label{sec:denseY}

We state two propositions in terms of the Young tableaux. Note that by taking $x=q+1$, we conclude that $B(n,m,y)$ is extremal in both cases.

\begin{prop}\label{prop:denseTxy}
    Fix integers $x>y>r>0$ and $m=xy-r$.
    Among all Young tableaux with $m$ cells in $[x]\times [y],$
    $h(T)$ is maximized by $$T'=\left([x]\times [y]\right) \backslash \left( x \times [y-r+1..y] \right).$$
    If $x>y=\omega(r^2)$, then the latter tableau has $h(T') \sim m \log m -r+o(1).$
\end{prop}

\begin{proof}
	    Let $T$ be a Young tableau which equals $[x]\times [y],$ except from the $r$ cells which have coordinates $(i_1,j_1),(i_2,j_2), \ldots, (i_r,j_r)$, where these are ordered in reverse lexicographic order. In particular $(i_1,j_1)=(x,y)$ and for every $1 \le k \le r$ the pair $(i_k,j_k)$ satisfies $i_k+j_k \ge x+y+1-k$, $j_k \le y$ and $i_k \le x.$
    	Since $\Delta$ is increasing and concave, by the inequality of Karamata, we have
    	$$ \Delta(i_k) + \Delta(j_k) \ge \Delta(x)+\Delta(y+1-k).$$
    	Now 
    	\begin{align*}
    	    h(T)&=xy \log(xy) - \sum_{k=1}^{r} \left( \Delta(i_k) + \Delta(j_k) \right)\\
    	    &\le xy \log(xy) - \sum_{k=1}^{r} \left( \Delta(x)+\Delta(y+1-k) \right)\\
    	    &= h(T'),
    	\end{align*}
    	and furthermore equality occurs if and only if $(i_k,j_k)=(x,y+1-k)$ for every $k$, i.e. $T=T'$.
    	Now since $1+\log(x)-\frac 1{x-1} \le \Delta(x) < 1+\log(x)$ and $1+\log(y) - \frac k{y-k})\le \Delta(y+1-k) < 1+\log(y)$, we note that when $y=\omega(r^2),$ we have
    	\begin{align*}
    	    h(T')&=xy \log(xy) - r \Delta(x)  -\sum_{k=1}^{r} \Delta(y+1-k)\\
    	    &= m\log(m+r)  -2r -o(1)\\
    	    &= m\log(m)-r-o(1),
    	\end{align*}
    	where the $o(1)$ term tends to zero as $x,y \to \infty$ for fixed $r$.
\end{proof}

\begin{prop}\label{prop:denseTbq}
	Fix integers $q\ge y>r>0$ and $m=qy+r$.
	Among all Young tableaux containing $[q]\times [y]$ with $m$ cells,
	$h(T)$ is maximized by $$T'=\left([q]\times [y]\right) \cup \left( (q+1) \times [r] \right).$$
	If $q=\omega(r)$, the latter tableau satisfies $h(T') \sim m\log (m)-r \log\left( \frac{y}{r} \right).$
\end{prop}

\begin{proof}
	Let $T$ be a Young tableau which equals $[q]\times [y],$ with $r$ additional cells which have coordinates $(i_1,j_1),(i_2,j_2), \ldots, (i_r,j_r)$, where these are ordered in lexicographic order. We note that for every $1 \le k \le r$ the pair $(i_k,j_k)$ satisfies $i_k+j_k \le q+k+1$ and $\min\{i_k,j_k\} \le k.$
    Since $\Delta$ is increasing and concave, by the inequality of Karamata, we have
    $$ \Delta(i_k) + \Delta(j_k) \le \Delta(\min\{i_k,j_k\})+\Delta(q+1+k-\min\{i_k,j_k\}) \le \Delta(q+1)+\Delta(k).$$
    	Now     	\begin{align*}
    	    h(T)&=yq \log(yq) + \sum_{k=1}^{r} \left( \Delta(i_k) + \Delta(j_k) \right)\\
    	    &\le yq \log(yq) + \sum_{k=1}^{r} \left( \Delta(q+1)+\Delta(k) \right)\\
    	    &= h(T'),
    	\end{align*}
    	and furthermore equality occurs if and only if $(i_k,j_k)=(q+1,k)$ for every $k$, i.e. $T=T'$.
    	
    	Now since $1+\log(q)\le \Delta(q+1) < 1+\log(q)+\frac 1q$ and $\sum_{k=1}^r \Delta(k)=f(r)=r \log(r)$, we note that when $q=\omega(r),$ we have
    	\begin{align*}
    	    h(T')&=yq \log(yq) + r \Delta(q+1)  +f(r)\\
    	    &=qy \log(q)+ qy\log(y)+r \log(q) +r+f(r) +o(1) \\
    	    &= m\log(m-r)-r \log(y)  +r+r \log(r) +o(1)\\
    	    &= m\log(m)-r\log\left( \frac yr \right)+o(1),
    	\end{align*}
    	where the $o(1)$ term tends to zero when $q \to \infty$ for fixed $r.$
\end{proof}

\section{Minimum is attained by dense Young tableaux}\label{sec:proofYdense}

In this subsection, we list and prove a number of lemmas that show that the extremal Young tableaux are dense, in the sense that they are of the form presented in Section~\ref{sec:denseY}.

We start with proving that the extremal Young tableau cannot have the following specific form.

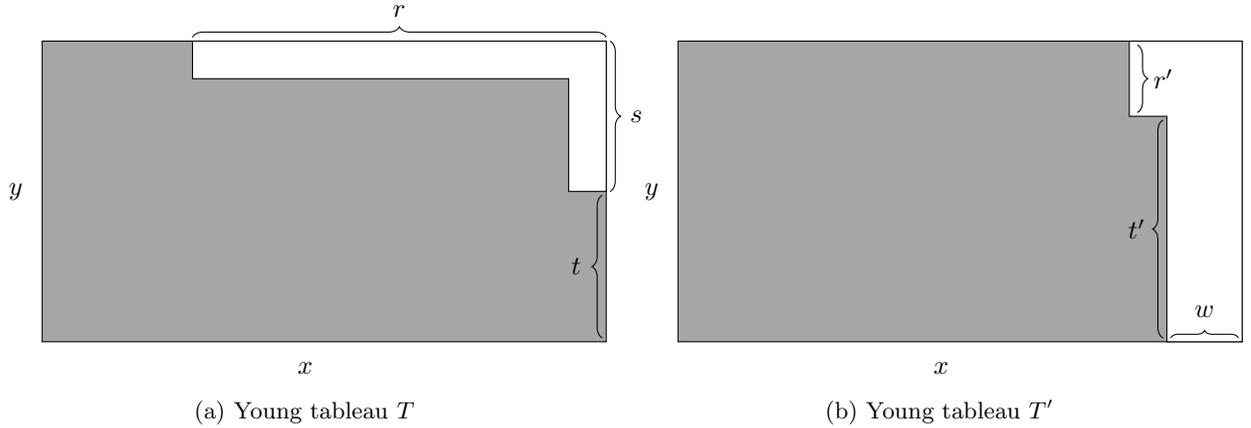
\begin{figure}[h]
\begin{minipage}[b]{.49\linewidth}
\begin{center}
\begin{tikzpicture}
\draw [fill = medium-gray](0,0) -- (7.5,0) -- (7.5,4) --(0,4) -- cycle;
\draw [decorate,decoration={brace,amplitude=4pt},xshift=0pt,yshift=0pt] (2,4.05) -- (7.5,4.05) node [black,midway,yshift=0.35cm]{$r$};
\draw [decorate,decoration={brace,amplitude=4pt},xshift=0pt,yshift=0pt] (7.55,4.) -- (7.55,2) node [black,midway,xshift=0.35cm]{$s$};
\draw [decorate,decoration={brace,amplitude=4pt},xshift=0pt,yshift=0pt] (7.45,0.05) -- (7.45,1.95) node [black,midway,xshift=-0.35cm]{$t$};
\draw (3.5,-0.35) node { $x$};
\draw (-.35,2) node {$y$};
\draw [fill = white] (2,3.5) -- (7,3.5) -- (7,2) -- (7.5,2)--(7.5,4)-- (2,4) -- cycle;
\end{tikzpicture}
\subcaption{Young tableau $T$}
\label{fig:T}
\end{center}
\end{minipage}\quad
\begin{minipage}[b]{.49\linewidth}
\begin{center}
\begin{tikzpicture}
\draw [fill = medium-gray](0,0) -- (7.5,0) -- (7.5,4) --(0,4) -- cycle;

\draw (3.5,-0.35) node { $x$};
\draw (-.35,2) node {$y$};
\draw [fill = white] (6.5,0) -- (7.5,0) -- (7.5,4) -- (6,4)--(6,3)--(6.5,3) -- cycle;
\draw [decorate,decoration={brace,amplitude=4pt},xshift=3pt,yshift=0pt] (6.,3.95) -- (6.,3.05) node [black,midway,xshift=0.35cm]{$r'$};
\draw [decorate,decoration={brace,amplitude=4pt},xshift=0pt,yshift=0pt] (6.45,0.05) -- (6.45,2.95) node [black,midway,xshift=-0.35cm]{$t'$};
\draw [decorate,decoration={brace,amplitude=0.1cm},xshift=0.cm,yshift=0.pt] (6.55,0.05) -- (7.45,0.05) node [black,midway,yshift=0.35cm]{$w$}; 
\end{tikzpicture}
\subcaption{Young tableau $T'$}
\label{fig:T'}
\end{center}
\end{minipage}
\caption{The Young tableaux from 
Lemma~\ref{lem:crucialcomparison}}\label{fig:Ywithdeletedhook}
\end{figure}

\begin{lem}\label{lem:crucialcomparison}
    Let $m=xy-r-s+1$ with $1\le r\le x-y$ and $1\le s<y$, such that $r+s-1=wy+r'$ for some integers $0<w$ and $0\le r' <y.$
    If
    \begin{align*}
        T&=\left([x]\times [y]\right) \backslash \left( [x-r+1..x] \times y \right) \backslash \left( x \times [y-s+1..y-1] \right) \mbox{ and}\\
        T'&=\left([x-w]\times [y]\right) \backslash \left( (x-w) \times [y-r'+1..y] \right)
        ,
    \end{align*}
    then
    $h(T) < h(T').$ See Figure~\ref{fig:Ywithdeletedhook}.
\end{lem}

\begin{proof}
    We prove this in a number of steps. First we prove that it is sufficient to prove when $s=1.$
    \begin{claim}
        If Lemma~\ref{lem:crucialcomparison} does hold whenever $s=1$, then so it does for all other cases.
    \end{claim}
    \begin{claimproof}
        Assume there are choices of $x,y,r,s$ (and thus Young tableaux $T$ and $T'$) for which Lemma~\ref{lem:crucialcomparison} is false.
        Let $t=y-s$ and $t'=y-r'$, and assume $0<t<y-1.$
        We now consider two cases.\\
        \textbf{Case 1: $t' \le t$.}
        In this case, we note that $\Delta(t+1)+\Delta(x)>\Delta(t'+1)+\Delta(x-w)$ and thus we also have a counterexample with $s-1$ instead of $s$.
        We can repeat this until $s=1.$\\
        \textbf{Case 2: $t' > t$.}
        Note that $y\cdot (x-w)>m>x\cdot (y-1)$ implying $\frac{x-w}{x}>\frac{y-1}{y} > \frac{t}{t+1} \ge \frac{t}{t'}$ and hence $(x-w)t'>xt.$
        Now we have $\Delta(t)+\Delta(x)<\Delta(t')+\Delta(x-w)$.
        When $t+x \le (x-w)+t'$, this is by Karamata's inequality and the fact that $\Delta$ is increasing and concave.
        When $t+x> (x-w)+t'$, we note that
        $(x-u)(t-u)=xt-u(x+t)+u^2< (x-w)t'-u((x-w)+t')+u^2=  (x-w-u)(t'-u)$ for every $u \ge 0$ and thus
        $$\int_{u=0}^{1} \log\left( (x-u)(t-u) \right) \d u < \int_{u=0}^{1} \log\left( (x-w-u)(t'-u) \right) \d u.$$
        The latter being equivalent to $\Delta(t)+\Delta(x)<\Delta(t')+\Delta(x-w)$.
        This implies that we can construct a counterexample with $s+1$ instead of $s$ as well and we can repeat this until $s=y,$
        which corresponds to an example with $s=1.$
    \end{claimproof}
    Next, we show that we can assume that $x-r=y.$
    
    \begin{claim}\label{clm:rmax}
        If Lemma~\ref{lem:crucialcomparison} does hold whenever $s=1$ and $r=x-y$, then so it does for all other cases.
    \end{claim}
    \begin{claimproof}
        Assume there are Young tableaux $T$ and $T'$ for which Lemma~\ref{lem:crucialcomparison} is false, where $s=1$ and $y+r<x.$
        Now deleting the first column of both tableaux, implies that 
        $h(T)$ decreases by $\Delta(x-r)+(y-1)\Delta(x)+f(y)$ and $h(T')$ has been decreased by $(y-r')\Delta(x-w)+r'\Delta(x-w-1)+f(y).$
        Since $\{\underbrace{x,x,\ldots, x}_{y-1},x-r\}$ majorizes $\{\underbrace{x-w,x-w,\ldots,x-w}_{y-r'},\underbrace{x-w-1,\ldots x-w-1}_{r'}\}$, as $\Delta$ is strictly concave, the value $h(T')$ has decreased by a larger amount than $h(T).$
        So Lemma~\ref{lem:crucialcomparison} is false for a construction with parameters $(x-1,y,r).$
        We can repeat this, until $x=y+r.$
    \end{claimproof}

Now we finish the proof by proving that Lemma~\ref{lem:crucialcomparison} is true when $s=1$ and $r=x-y.$
In the latter case, we can write $x=ay+b$ where $a>0$ is an integer and $0 \le b <y$, in which case we know the precise shapes of $T$ and $T'$. These have been represented in Figure~\ref{fig:Yremainingcase}.

\begin{figure}[h]
\begin{minipage}[b]{.49\linewidth}
\begin{center}
\begin{tikzpicture}
\draw [fill = medium-gray](0,0) -- (7.5,0) -- (7.5,2.5) --(3,2.5) --(3,3)--(0,3)-- cycle;
\draw [decorate,decoration={brace,amplitude=4pt},xshift=0pt,yshift=0pt] (0.02,3.05) -- (2.98,3.05) node [black,midway,yshift=0.35cm]{$y$};
\draw (3.75,-0.35) node { $ay+b$};
\draw (-.35,2) node {$y$};
\end{tikzpicture}
\subcaption{Young tableau $T$}
\end{center}
\end{minipage}
\quad
\begin{minipage}[b]{.49\linewidth}
\begin{center}
\begin{tikzpicture}
\draw [fill = medium-gray](0,0) -- (7,0) -- (7,1.5) --(6.5,1.5) --(6.5,3)--(0,3)-- cycle;
\draw (3.5,-0.35) node { $a(y-1)+b+1$};
\draw (-.35,2) node {$y$};
\draw [decorate,decoration={brace,amplitude=4pt},xshift=0pt,yshift=0pt] (6.55,2.95)--(6.55,1.55)  node [black,midway,xshift=0.35cm]{$b$};
\end{tikzpicture}
\subcaption{Young tableau $T'$}
\end{center}
\end{minipage}
\caption{The two Young tableaux in the remaining case }\label{fig:Yremainingcase}
\end{figure}
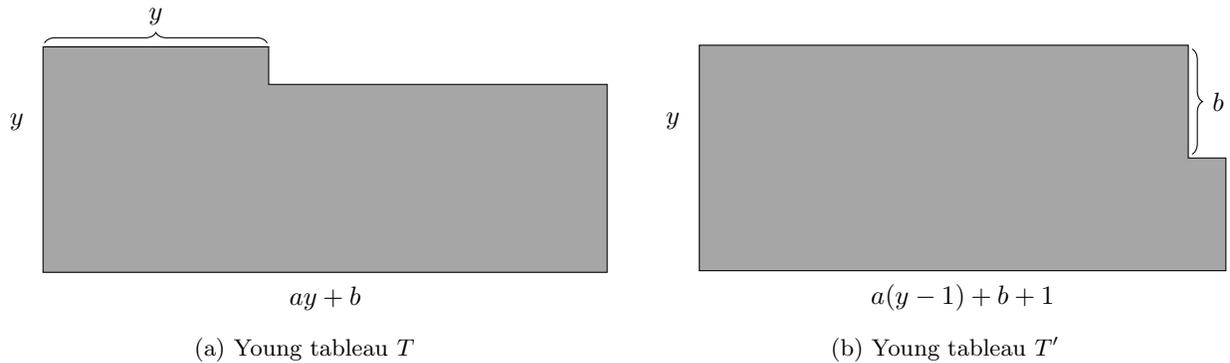

When $a=1$, we have $h(T)\le h(T')$ by Proposition~\ref{prop:denseTxy} and the inequality is even strict when $b>1.$
We now finish the proof by induction on $a$. So assume it is proven for the values $a,b,y.$
Going from $a$ to $a+1$, the value $h(T)$ increases by 
\begin{equation}\label{eq:T_a_a+1}
    I_1=y \cdot f(y-1) + (y-1) \left( f((a+1)y+b)-f(ay+b) \right)
\end{equation}
while $h(T')$ increases by at least
\begin{equation}\label{eq:T'_a_a+1}
    I_2= (y-1) \cdot f(y)+ y \left( f\left(\left(a+1+\frac{b}{y}\right)(y-1)\right)-f\left(\left(a+\frac{b}{y}\right)(y-1)\right) \right).\\
\end{equation}
Now one can note that $I_1=I_2$ by implementing the following two equalities in Equation~\eqref{eq:T_a_a+1} and Equation~\eqref{eq:T'_a_a+1}.
\begin{align*}
f((a+1)y+b)-f(ay+b)&= y +y \int_{u=a+\frac by}^{a+1+\frac by} \left(\log(u)+\log(y)\right) \d u\\
\end{align*}
\begin{align*}
f\left(\left(a+1+\frac{b}{y}\right)(y-1)\right)-f\left(\left(a+\frac{b}{y}\right)(y-1)\right)
&= y-1+(y-1)\int_{u=a+\frac by}^{a+1+\frac by} \left(\log(u)+\log(y-1)\right) \d u.
\end{align*}
\end{proof}

Having proven this particular case in Lemma~\ref{lem:crucialcomparison}, we continue with some observations from which we can conclude that the extremal tableaux will be as given in Section~\ref{sec:denseY}. 

\begin{lem}\label{lem:imbalance}
    If $G=(U \cup V,E)$ is an extremal bipartite $(n,m)$-graph, then the minimum degree of the smaller partition class is at least the maximum degree of the other partition class.
\end{lem}

\begin{proof}
    Consider the associated Young tableau $T$ of $G$.
    Let $i$ be the value for which $(i,i) \in T$ and $(i+1,i+1) \not \in T$.
    Let $S_1=\{(i',j') \in T \mid j' >i \}$ and $S_2=\{(i',j') \in T \mid i' >i \}.$
    Now if both $S_1$ and $S_2$ are non-empty, we can construct a tableau $T'$ for which the rows in $S_1$ are deleted and are added as columns, in such a way that another Young tableau $T'$ has been formed.
    It is not hard to see that the degree sequence of the graph $G'$ associated with $T'$ majorizes the degree sequence of $G$. For this, note that is sufficient to compare the degrees of $\{u_1, \ldots, u_i,v_1, \ldots, v_i\}$ (and thus the number of cells in their rows/columns).
    Thus, since $f$ is strictly convex, by Karamata's inequality we know $h(T')>h(T).$
\end{proof}

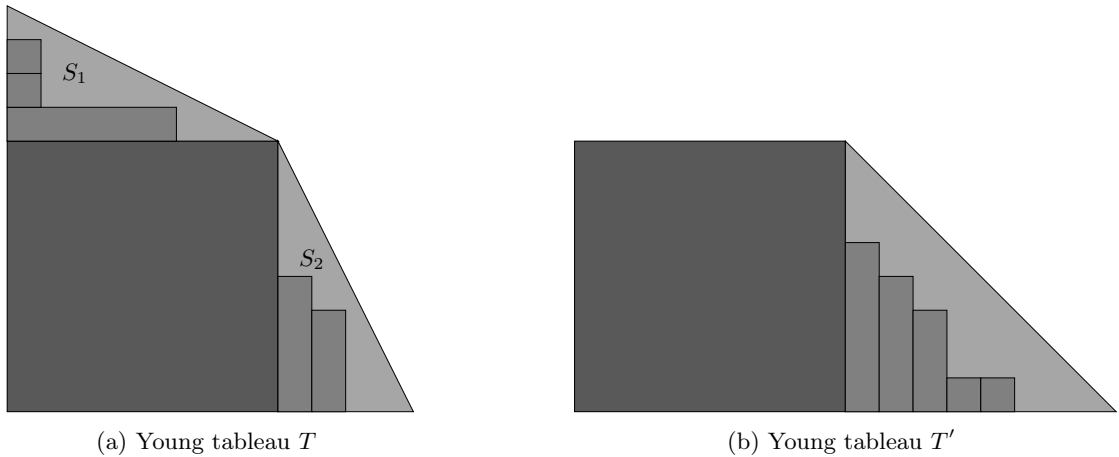
\begin{figure}[h]
\begin{minipage}[b]{.49\linewidth}
\begin{center}
\scalebox{0.9}{
\begin{tikzpicture}
\draw [fill = dark-gray](0,0) -- (4,0) -- (4,4) --(0,4) -- cycle;
\draw [fill = medium-gray](6,0) -- (4,0) -- (4,4) -- cycle;
\draw [fill = medium-gray](0,6) -- (0,4) -- (4,4) -- cycle;
\draw [fill = half-gray](0,4.5) -- (2.5,4.5) -- (2.5,4) --(0,4) -- cycle;
\draw [fill = half-gray](0,4.5) -- (0.5,4.5) -- (0.5,5) --(0,5) -- cycle;
\draw [fill = half-gray](0,5.5) -- (0.5,5.5) -- (0.5,5) --(0,5) -- cycle;
\draw [fill = half-gray](4.5,2) -- (4.5,0) -- (4,0) --(4,2) -- cycle;
\draw [fill = half-gray](4.5,1.5) -- (4.5,0) -- (5,0) --(5,1.5) -- cycle;
\draw (1,5) node { $S_1$};
\draw (4.5,2.25) node {$S_2$};
\end{tikzpicture}}\\
\subcaption{Young tableau $T$}
\label{fig:T}
\end{center}
\end{minipage}\quad
\begin{minipage}[b]{.49\linewidth}
\begin{center}
\scalebox{0.9}{
\begin{tikzpicture}
\draw [fill = dark-gray](0,0) -- (4,0) -- (4,4) --(0,4) -- cycle;
\draw [fill = medium-gray](8,0) -- (4,0) -- (4,4) -- cycle; 
\draw [fill = half-gray](4.5,2.5) -- (4.5,0) -- (4,0) --(4,2.5) -- cycle;
\draw [fill = half-gray](4.5,2) -- (4.5,0) -- (5,0) --(5,2) -- cycle;
\draw [fill = half-gray](5.5,1.5) -- (5.5,0) -- (5,0) --(5,1.5) -- cycle;
\draw [fill = half-gray](5.5,0.5) -- (5.5,0) -- (6,0) --(6,0.5) -- cycle;
\draw [fill = half-gray](6.5,0.5) -- (6.5,0) -- (6,0) --(6,0.5) -- cycle;
\end{tikzpicture}}
\subcaption{Young tableau $T'$}
\label{fig:T'}
\end{center}
\end{minipage}
\caption{Sketch of the rearrangement in Lemma~\ref{lem:imbalance}}
\end{figure}

\begin{thr}\label{thr:finalstep_Bnmy}
For fixed integers $m$ and $n$, let $T_{*}$ be a tableau with $m$ cells for which the sum of its length $x$ and width $y$ is at most $n$. If $h(T_{*})$ is maximal among all such tableaux, then $x,y$ satisfy $xy-\min\{x,y\}<m \le xy.$
\end{thr}

\begin{proof}
    Assume $m$ cannot be written as a product $xy$ with $x+y\le n$ and 
    $T_{*}$ is an extremal tableau with $m\le (x-1)y,$ where $y <x.$
    Let $x_1=x, x_2, \ldots, x_y$ be the number of cells in every row.
    By Lemma~\ref{lem:imbalance}, we know that $[y]\times [y] \subset T_{*}$ and thus $x_i \ge y$ for every $1 \le i \le y.$
    Furthermore, since $m\le (x-1)y,$ there is a smallest index $i$ such that $x_i<x-1.$
    Now, the first $i$ rows of $T_{*}$, i.e. $T_i=T_{*} \cap \left([x]\times [i]\right)$, form a tableau that has the form of the tableau $T$ in Lemma~\ref{lem:crucialcomparison}.
    Replacing $T_i$ by the corresponding $T'_i$ from Lemma~\ref{lem:crucialcomparison} (which does not increase the length of the tableau) will imply that we also form a $T'_{*}$ for which $h(T'_{*})>h(T_{*}),$ since $h(T'_{*})-h(T_{*})=h(T'_i)-h(T_i)>0.$
    For the latter, note that the only rows and columns that possibly changed, are rows $1$ till $i$ and columns $x_i+1$ till $x,$ all of which do not contain any cell outside $T_i$ or $T'_i$.
    This is the desired contradiction.
    Figure~\ref{fig:Youngtableau_conclusion} presents this final comparison.
\end{proof}

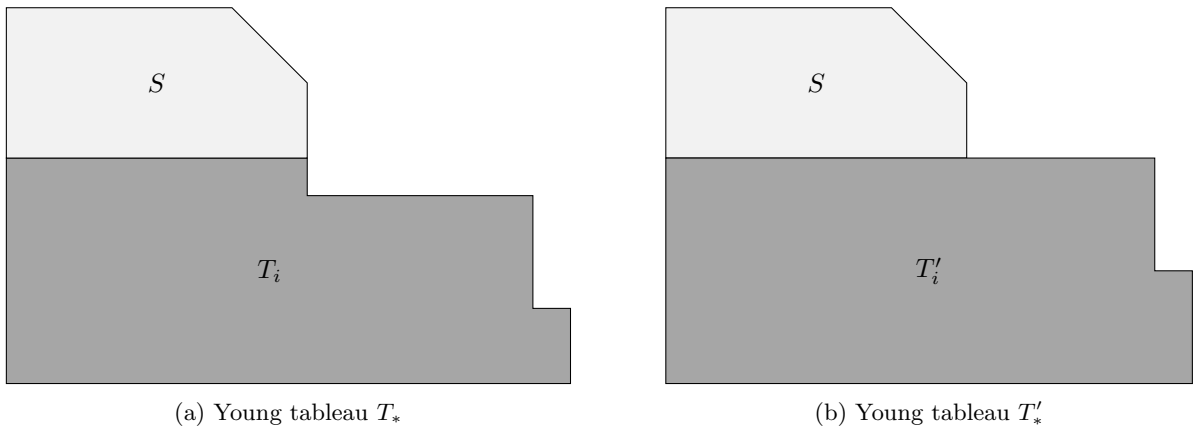
\begin{figure}[h]
\begin{minipage}[b]{.49\linewidth}
\begin{center}
\begin{tikzpicture}
\draw [fill = medium-gray](0,0) -- (7.5,0)--(7.5,1) -- (7,1)--(7,2.5) --(4,2.5) --(4,3)--(0,3)-- cycle;
\draw [fill = light-gray](0,5)--(3,5)--(4,4)--(4,3)--(0,3)-- cycle;
\draw (2,4) node {$S$};
\draw (3.5,1.5) node {$T_i$};
\end{tikzpicture}
\subcaption{Young tableau $T_{*}$}
\end{center}
\end{minipage}
\quad
\begin{minipage}[b]{.49\linewidth}
\begin{center}
\begin{tikzpicture}
\draw [fill = light-gray](0,5)--(3,5)--(4,4)--(4,3)--(0,3)-- cycle;
\draw [fill = medium-gray](0,0) -- (7,0) -- (7,1.5) --(6.5,1.5) --(6.5,3)--(0,3)-- cycle;
\draw (2,4) node {$S$};
\draw (3.5,1.5) node {$T'_i$};
\end{tikzpicture}
\subcaption{Young tableau $T'_{*}$}
\end{center}
\end{minipage}
\caption{The local move increasing $h(T_{*})$ in Theorem~\ref{thr:finalstep_Bnmy} }\label{fig:Youngtableau_conclusion}
\end{figure}

\section{Maximum entropy and extrema of other chemical indices}\label{sec:chemicalindex}

The degree-based entropy is one example of a degree-based topological index (or chemical index), the latter defined as
$$\TI_g(G) = \sum_{uv \in E(G)} g(\deg(u), \deg(v)),$$
for some function $g$, as initially done in e.g.~\cite{gutman2013degree}.
See e.g.~\cite[Table~1]{LP22} for a clear overview of some more examples. 
In this case, one can conclude that the interesting graphical function-indices, defined similarly as $\sum_{v \in V(G)} g(\deg(u))$, will essentially boil down to $g(x)=c_1x+c_2 x \log(x)$ and thus be directly related to the entropy case. So we focus on the chemical indices.
The function $h(G)$ can be written as $\TI_g$ where $g(x,y)=\log(xy)$.
Our idea of approach $2$ in Subsection~\ref{subsec:approach2} works for any $\TI_g$ whenever $g$ is an increasing, concave function in $xy$.
E.g. when $g(x,y)=xy$ or $g(x,y)=\sqrt{xy}$ (i.e. for the Second Zagreb index and Reciprocal Randic index), the graphs maximizing $\TI_g(G)$ among all bipartite graphs of size $m$ are precisely the complete bipartite graphs $K_{q,y}$ with $q \cdot y=m$. 
When we know both the order and size, one can expect that a similar exposition implies that the extremal graphs are again near complete bipartite, i.e. $K_{y,x}$ with some small number $r$ of removed (or added) edges. Xu et al.~\cite{XTLW15} studied these for the 
first and second 
Zagreb indices for example.

On the other hand, if $g(x,y)=x+y$ or $g(x,y)=(x+y)^2$, since $\deg(u)+\deg(v) \le m+1$ for every $uv \in E(G)$ the unique extremal graph (maximizing $\TI_g(G)$) is easily seen to be the star $K_{1,m}.$
In particular, this implies that among bipartite graphs of a given size, the set of graphs maximizing the First Zagreb index and Second Zagreb index are not equal. This contrasts some of the intuition in the concluding section of~\cite[Sec.~4]{XTLW15}. 
Also among bipartite graphs of a given order and size,
the bipartite graphs with the maximum first Zagreb index
might be different than the one with the maximum second Zagreb index.
For the first Zagreb index these extremal bipartite graphs were determined by Zhang and Zhou~\cite{Zhang2014}. Here the asymmetry (large degrees) play a major role and so there are examples for which the bipartite graphs maximizing the first Zagreb index are not complete bipartite, while there do exist complete bipartite graphs $K_{y,x}$ with $xy=m$ and $x+y<n$ that maximize the second Zagreb index.

Finally, we prove that the graphs with maximum entropy can be easily determined and are the graphs for which the degree sequence is as balanced as possible.
We prove these statements by showing that the latter graphs are those which minimize $\TI_g(G)$ for any increasing, strictly convex function $g$.
Thus, we give a short alternative proof for the result in~\cite{Dong21} in this more general setting and for the result of~\cite{HLP22}.

Among all graphs with fixed size $m$, without constraint on the order, $mK_2$ is the unique extremal graph maximizing the entropy (note that $h(mK_2)=0)$ and minimizing $\TI_g(G)$ as well (here the order $n \ge 2m$ and isolated vertices are taken into account as well when $g(0) \not=0$).

If there is a condition on both the order and size, the extremal graphs and bipartite graphs are precisely the nearly-regular graphs, as proven in the following two propositions.
Cheng et al.~\cite{CGZD09} determined these bipartite graphs with the minimum value of the first Zagreb index for bipartite graphs of a given order and size.
The proof is essentially a corollary of Karamata's inequality and noting that there is a degree sequence that is majorized by all other possible realizations.

\begin{prop}
    Let $g$ be an increasing, strictly convex function.
    Among all graphs with fixed size $m$ and order $n$, every nearly-regular graph (i.e. $\deg_{max}-\deg_{min} \le 1$) is extremal and these are the only extremal graphs.
\end{prop}
\begin{proof}
     Since the degree sequence of any nearly-regular graph is majorized by the degree sequence of any  other $(n,m)$-graph, the result follows by Karamata's inequality. The characterization of the extremal graphs is a consequence of $g$ being strictly convex.
     The existence of nearly-regular graph (even connected if $m\ge n-1$ when necessary) is immediate.
     If $n$ is even, then $K_n$ can be partitioned into $n-1$ perfect matchings and one can add edges from one perfect matching at a time.
     If $n$ is odd, then $K_n$ can be partitioned into $\frac{n-1}{2}$ $2$-regular graphs, where at least one of them is a cycle (when $n=2k+1$ and $V=\{v_1,v_2,\ldots, v_n\}$, for every $1 \le \ell \le k$, form the graph with edges $v_iv_j$ for which $ i-j \equiv \pm \ell \pmod n$).
     Now when $m=an+b,$ one can take the union of $a$ such $2$-regular graphs and a matching or its complement of the cycle.
\end{proof}

\begin{prop}
    Among all bipartite graphs with size $m$ and order $n$, every bipartite graph $G=(U \cup V,E)$ with $\lvert U \rvert = \left\lceil \frac n2 \right\rceil$ and $\lvert V \rvert = \left\lfloor \frac n2 \right\rfloor$ for which the degrees in one partition class differ by at most one attains the maximum entropy in the graph class. Furthermore, a bipartite graph is extremal if and only if it 
    has the same degree sequence as this balanced example.
\end{prop}

\begin{proof}
Let $G'=(U' \cup V',E')$ be any other bipartite graph with size $m$ and order $n$ whose partition classes have size $u \ge v$.
The sum of the $i \le v$ largest degrees in $V'$ are at least equal to the sum of the $i$ largest degrees of $G$ and the sum of the $i\le u$ smallest degrees in $U'$ are at most the sum of the $i$ smallest degrees in $G$. Here we use that $u \ge \left\lceil \frac n2 \right\rceil $ and $v \le \left\lfloor \frac n2 \right\rfloor$ and the degrees in $G$ are as balanced as possible.
Remark that if $u \ge i >\left\lceil \frac n2 \right\rceil $, then the sum of the $i$ degrees in $U'$ is at most $m$.
So we conclude that the degree sequence of $G$ is majorized by any degree sequence of any other bipartite graph with order $n$ and size $m$, from which we conclude (also for  the uniqueness statement) by Karamata's inequality and $g(x)$ being strictly convex.

We remark that one can always construct at least one such a balanced bipartite graph $G$.
If $n$ is even, just partition $K_{n/2,n/2}$ in perfect matchings and add the edges from one matching at a time up to the point you selected precisely $m$ edges.
When $n$ is odd, for every $1 \le m \le \left\lfloor \frac{n^2}{4} \right\rfloor,$ one can construct a graph by adding for every $1 \le k \le m$ an edge between the vertices $a_i$ and $b_j$ for which $k \equiv i \pmod{ \left\lceil \frac n2 \right\rceil}$ and $k \equiv j \pmod{ \left\lfloor \frac n2 \right\rfloor}$ respectively.
\end{proof}

A few examples have been presented in Figure~\ref{fig:maxentropy_bipgraphs}.
Since everything boils down to having a degree sequence that is as balanced as possible, i.e. it is majorized by any other degree sequence of a graph, the result only depends on the degree sequence. 

We end with a few observations.
There do exist both connected and disconnected extremal bipartite graphs (as presented for $m=n=9$).
If $n$ is odd, it is possible there is no extremal bipartite graph attaining the maximum over all graphs, since for bipartite $(n,m)$ graphs we might be forced to have $\deg_{max}-\deg_{min} \ge 2$ (equality can be attained), as is the case with $m=n=9$.
Finally, we observe that the partition itself does not necessarily has to be balanced (e.g. $m=30, n=22, \lvert U\rvert =10$ and $\lvert V\rvert =12$) .

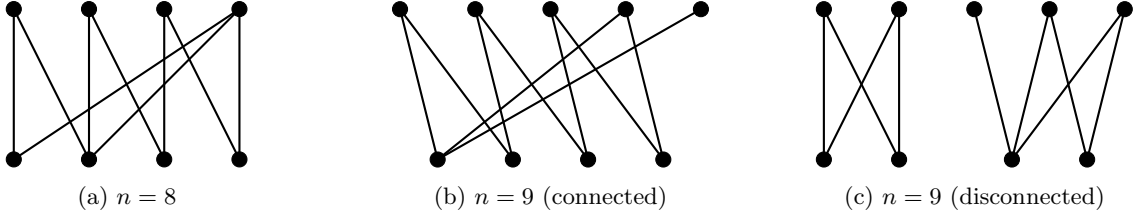
\begin{figure}[h]
\begin{minipage}[b]{.32\linewidth}
\begin{center}
    \begin{tikzpicture}
    {
    
	\foreach \x in {1,2,3,4}{\draw[thick] (\x,2) -- (\x,0);}
	\foreach \x in {1,2,3}{\draw[thick] (\x,2) -- (\x+1,0);}
	\draw[thick] (4,2) -- (2,0);
	\draw[thick] (1,0) -- (4,2);
	\foreach \x in {1,2,3,4}{\draw[fill] (\x,0) circle (0.1);\draw[fill] (\x,2) circle (0.1);}	
	}
	\end{tikzpicture}\\
	\subcaption{$n=8$}\label{fig:1a}
\end{center}
\end{minipage}\quad\begin{minipage}[b]{.32\linewidth}
\begin{center}
     \begin{tikzpicture}
    {	
	\foreach \x in {1,2,3,4}{\draw[thick] (\x,5) -- (\x+0.5,3);}
	\foreach \x in {1,2,3}{\draw[thick] (\x,5) -- (\x+1.5,3);}
	\draw[thick] (5,5) -- (1.5,3);
	\draw[thick] (4,5) -- (1.5,3);
	\foreach \x in {1,2,3,4,5}{\draw[fill] (\x,5) circle (0.1);}
	\foreach \x in {1,2,3,4}{\draw[fill] (\x+0.5,3) circle (0.1);}
    }
    \end{tikzpicture}\\
    \subcaption{$n=9$ (connected)}\label{fig:1b}
\end{center}
\end{minipage}\quad\begin{minipage}[b]{.32\linewidth}
\begin{center}
     \begin{tikzpicture}
    {	
	\foreach \x in {1,2}{\foreach \y in {1,2}{\draw[thick] (\x,5) -- (\y,3);}}
	\foreach \x in {3,4,5}{\draw[thick] (\x,5) -- (3.5,3);}
	\foreach \x in {4,5}{\draw[thick] (\x,5) -- (4.5,3);}
	\foreach \x in {1,2,3,4,5}{\draw[fill] (\x,5) circle (0.1);}
	\foreach \x in {1,2}{\draw[fill] (\x,3) circle (0.1);}
	\foreach \x in {3,4}{\draw[fill] (\x+0.5,3) circle (0.1);}
    }
    \end{tikzpicture}\\
    \subcaption{$n=9$ (disconnected)}\label{fig:1c}
\end{center}
\end{minipage}
\caption{Bipartite graphs of size $9$ with maximum entropy}\label{fig:maxentropy_bipgraphs}
\end{figure}

\section{Conclusion}\label{sec:conc}

In this paper, we proved that the bipartite graphs of size $m$ that maximize the entropy are complete bipartite graphs.
If there is a restriction on the order (at most $n$) and $m$ cannot be written as $x\cdot y$ where $x+y \le n$, then the extremal graphs are of the form $B(n,m,y),$ i.e. a complete bipartite graph minus a star.
Nevertheless, it seems impossible to determine $y$ with a general formula, as we will explain here.

Using a computer program\footnote{\url{https://github.com/MatteoMazzamurro/extrema-graph-entropy/blob/main/bipartite_graphs_entropy.R}}, we could find all extremal bipartite graphs for $ n \le 50$ and $m \le \left\lfloor \frac{n^2}{4} \right\rfloor.$
The output is presented in\footnote{\url{https://github.com/MatteoMazzamurro/extrema-graph-entropy/blob/main/B_n_m_50.csv}}.
For every pair $(n,m)$ and value $1 \le y \le \sqrt{m}$, we computed the values     $$q=\left\lfloor\frac{m}{y}\right\rfloor,\;  x=\left\lceil\frac{m}{y}\right\rceil,\; m-qy,\; xy-m,\; h=h(B(n,m,y))$$
and three boolean values expressing whether $B(n,m,y)$ achieved the maximum value of $h$ among all choices of $y$ and whether $m-qy$ or $xy-m$ were minimal among all possible choices for $y$ for that pair $(n,m).$
Among all extremal graphs which are not complete bipartite, we found $547$ cases for which both $xy-m$ and $m-qy$ are minimal among all possibilities, $375$ cases where only $m-qy$ is minimal, $3635$ cases where only $xy-m$ is minimal, and no cases for which neither $xy-m$ nor $m-qy$ is minimal.

Intuitively, when $xy-m$ or $m-qy$ is small, the graph $B(n,m,y)$ is nearly complete bipartite (up to the edges of a small star which has been added or removed) and so it might be extremal.
Nevertheless, this was found not to be the case in general.
When $n=17726$ and $m=318728$, the extremal graph is $B(n,m,18)$, which satisfies $q=17707$, $y=18$ and $r=m-qy=2,$ even though $y=139$ and $q=2293$ would give $r=1$.
Taking into account the estimates in Proposition~\ref{prop:denseTbq}, one can expect that there are more such examples since the difference with the upper bound depends on both $r$ and $y$, i.e. one would like $y$ and $r$ to be both small and the ratio has an influence on this. A precise statement here would be difficult, as it is related to some hard number theoretic questions.

Comparing with Proposition~\ref{prop:denseTxy}, whenever $m$ is large, one would wish that there was a small value of $r$ for which $m+r$ could be written as a product $xy$ with $x+y \le n.$
So one may expect from these observations that in most cases, the extremal graph will be $B(n,m,y)$ with $y$ chosen in such a way that $y\left\lceil \frac m y \right\rceil -m$ is minimized under the constraint that $y+\left\lceil \frac m y \right\rceil \le n,$ and so the extremal graph is a ``complete bipartite graph minus a small star".
Furthermore, when there are multiple choices for $y$, we need to choose the smallest $y$, i.e. for fixed $r$, the complete bipartite graph will be as asymmetric as possible.
To note the latter, one can compare the computations in Proposition~\ref{prop:denseTxy} for some $x'>x>y>y'$ with $xy=x'y'$.
For a fixed $1 \le k \le y-1$, we have $\Delta(x)+\Delta(y+1-k) \ge \Delta(x')+\Delta(y'+1-k)$ since $(x-u)(y+1-k-u) \ge (x'-u)(y'+1-k-u)$ for $u \in [0,1].$

We can prove that the intuition that $r=xy-m$ is minimal for the extremal $B(n,m,y)$ is true in many cases, using the notion of $M$-friable numbers.
A positive integer is called $M$-smooth or $M$-friable if its largest prime divisor is at most $M.$
At this point, the state-of-the-art on smooth/friable numbers~\cite{MR16_smooth, E17_smooth} does not address the question if for fixed constants $0<u<1$ and $c>0$, the interval $[m,m+c \log(m)]$ necessarily needs to contain a $m^u$-smooth number whenever $m$ is sufficiently large, but it does so for almost all $m$ by~\cite[Cor.~6]{MR16_smooth}.

Write $n=(m+\log(m))^u+(m+\log(m))^{1-u}$ for some real number $\frac 23 \le  u < 1$.
Now for almost all $m$, we know that there does exist a number $r \le 0.5 (1-u) \log(m)$ such that $m+r$ is $m^u$-smooth (by~\cite[Cor.~6]{MR16_smooth}).
Now $m+r$ being $m^u$-smooth implies that all its prime divisors are bounded by $m^u$ and thus also $(m+r)^u.$
Thus $m+r$ has two divisors $x,y$ for which $(m+r)^{1-u}\le y \le x \le  (m+r)^u$.
For this, note that if the largest prime factor $p$ of $m+r$ satisfies $(m+r)^{1-u}\le p \le (m+r)^u$, we can take $\{x,y\}$ to be equal to $\left\{p, \frac {m+r}p \right\}$ and otherwise we can take a product $P$ of prime divisors of $m+r$ which is between $(m+r)^{1-u}$ and $(m+r)^{2(1-u)} \le (m+r)^u$ and take $\{x,y\}$ to be equal to $\left\{P, \frac {m+r}P \right\}$.
Since the function $P+\frac{m+r}{P}$ is increasing for $P \ge \sqrt{m+r}$ and $r < \log m$, we have $x+y <n.$
Comparing Proposition~\ref{prop:denseTbq} and Proposition~\ref{prop:denseTxy} we conclude that the extremal graph is equal to the graph $B(n,m,y)$ where $y$ has been chosen in such a way that $xy-m$ is minimized.
The latter is because $y+\frac{m}{y} \le n$ and $y< \sqrt{m}$ implies that $y>> m^{0.5(1-u)}$ and thus $\log(y)-r$ is large.

On the other hand, there will be infinity many examples which are not of this form.
A simple example for this, where $m-qy$ can be arbitrarily large, can be constructed as follows.
Fix $c>0$ and let $n=2a+1$ and $m=a^2+c$,
where $a=b^2+b$.
Now $B(n,m,a)$ is an example for which $m-qy=c$ is minimized (note that $q+y=2a$ is needed).
On the other hand $B(n,m,y)$ will always satisfy $xy-m \ge 2b-c$ since (note that $x+y=n$ is needed) $$(a-b+1)(a+b)-m=a^2+a-b^2+b>m>a^2+a-b^2-b=(a-b)(a+b+1).$$
For fixed $c$ and $b$ sufficiently large, we conclude as then $2b-c>c \log\left( \frac{2b-c}{c} \right).$

\section*{Acknowledgement}
Thanks to Bart Michels and Régis de la Bretèche for sharing their expertise on smooth/friable numbers.

\paragraph{Open access statement.} For the purpose of open access,
a CC BY public copyright license is applied
to any Author Accepted Manuscript (AAM)
arising from this submission.

\bibliographystyle{abbrv}
\bibliography{bib_entropy}

\end{document}